\documentclass[11pt]{article}
\usepackage{amssymb}
\newtheorem{precor}{{\bf Corollary}}

\newenvironment{cor}{\begin{precor}{\hspace{-0.5
               em}{\bf.\ }}}{\end{precor}}
\newtheorem{precon}{{\bf Conjecture}}

\newtheorem{prealphcon}{{\bf Conjecture}}

\newtheorem{predefin}{{\bf Definition}}

\newtheorem{preexm}{{\bf Example}}

\newtheorem{preappl}{{\bf Application}}

\newtheorem{prelem}{{\bf Lemma}}

\newtheorem{preproof}{{\bf Proof.\ }}

\newenvironment{proof}[1]{\begin{preproof}{\rm
               #1}\hfill{$\blacksquare$}}{\end{preproof}}
\newtheorem{pretheorem}{{\bf Theorem}}

\newenvironment{theorem}{\begin{pretheorem}{\hspace{-0.5
               em}{\bf.\ }}}{\end{pretheorem}}
\newtheorem{prealphtheorem}{{\bf Theorem}}

\newtheorem{prealphlem}{{\bf Lemma}}

\newtheorem{prepro}{{\bf Proposition}}

\newtheorem{preprb}{{\bf Problem}}

\newtheorem{prerem}{{\bf Remark}}

\newtheorem{preapp}{{\bf Application}}

\newtheorem{prequ}{{\bf Question}}

\def\conct[#1,#2]{\mbox {${#1} \leftrightarrow {#2}$}}
\def\dconct[#1,#2]{\mbox {${#1} \rightarrow {#2}$}}
\def\deg[#1,#2]{\mbox {$d_{_{#1}}(#2)$}}
\def\mindeg[#1]{\mbox {$\delta_{_{#1}}$}}
\def\maxdeg[#1]{\mbox {$\Delta_{_{#1}}$}}
\def\outdeg[#1,#2]{\mbox {$d_{_{#1}}^{^+}(#2)$}}
\def\minoutdeg[#1]{\mbox {$\delta_{_{#1}}^{^+}$}}
\def\maxoutdeg[#1]{\mbox {$\Delta_{_{#1}}^{^+}$}}
\def\indeg[#1,#2]{\mbox {$d_{_{#1}}^{^-}(#2)$}}
\def\minindeg[#1]{\mbox {$\delta_{_{#1}}^{^-}$}}
\def\maxindeg[#1]{\mbox {$\Delta_{_{#1}}^{^-}$}}

\def\dre[#1,#2,#3]{\mbox {${\cal E}^{^{#3}}(#1,#2)$}}
\def\var[#1,#2]{\mbox {${\rm Var}_{_{#1}}(#2)$}}
\def\ls[#1]{\mbox {$\xi^{^{#1}}$}}
\def\hom[#1,#2]{\mbox {${\rm Hom}({#1},{#2})$}}
\def\onvhom[#1,#2]{\mbox {${\rm Hom^{v}}(#1,#2)$}}
\def\onehom[#1,#2]{\mbox {${\rm Hom^{e}}(#1,#2)$}}
\def\core[#1]{\mbox {$#1^{^{\bullet}}$}}
\def\cay[#1,#2]{\mbox {${\rm Cay}({#1},{#2})$}}
\def\sch[#1,#2,#3]{\mbox {${\rm Sch}({#1},{#2},{#3})$}}
\def\cays[#1,#2]{\mbox {${\rm Cay_{s}}({#1},{#2})$}}
\def\dirc[#1]{\mbox {$\stackrel{\rightarrow}{C}_{_{#1}}$}}
\def\cycl[#1]{\mbox {${\bf Z}_{_{#1}}$}}

\begin{document}

\begin{center} 
{\Large \bf On The Circular Altitude of Graphs}\\
\vspace{0.3 cm}
{\bf Saeed Shaebani}\\
{\it School of Mathematics and Computer Science}\\
{\it Damghan University}\\
{\it P.O. Box {\rm 36716-41167}, Damghan, Iran}\\
{\tt shaebani@du.ac.ir}\\ \ \\
\end{center}
\begin{abstract}
\noindent Peter Cameron introduced the concept of the circular altitude of graphs \cite{cameron};
a parameter which was shown by Bamberg et al. \cite{bamberg} that provides a lower bound
on the circular chromatic number. In this note, we investigate this parameter and show that the
circular altitude of a graph is equal to the maximum of circular altitudes of its blocks.
Also, we show that homomorphically equivalent graphs have the same circular altitudes.
Finally, we prove that the circular altitude of the Cartesian product of two graphs is equal
to the maximum of circular altitudes of its factors.
\\


\end{abstract}
\section{Introduction}

Unless otherwise stated, considered graphs in this paper are simple
and also have nonempty and finite vertex sets. Also,
the vertex set and the edge set of a graph $G$
are denoted by the symbols $V(G)$ and $E(G)$, respectively.

Let $G$ be a graph and $\mathcal{U}:\{1,2,\ldots , |V(G)|\}  \rightarrow V(G)$ be a bijection.
For simplicity, for each $i$ in $\{1,2,\ldots , |V(G)|\}$, we write $u _{i}$ instead of
$\mathcal{U}(i)$.
By a {\it $\mathcal{U}$-monotonic cycle} of length $m$, we mean a sequence
$u_{i_{1}},u_{i_{2}},\ldots , u_{i_{m}}$ of $m$ pairwise distinct vertices of $G$
such that whenever $m>1$, the following three conditions hold simultaneously :
\begin{itemize}
\item $i_{1} < i_{2} < \cdots < i_{m}$;
\item For each $t$ with $1\leq t \leq m-1$, two vertices $u_{i_{t}}$
and $u_{i_{t+1}}$ are adjacent;
\item Two vertices $u_{i_{m}}$ and $u_{i_{1}}$ are adjacent.
\end{itemize}
One can consider the bijection $\mathcal{U}$ as a linear ordering
of $V(G)$ like $u_{1},u_{2},\ldots , u_{|V(G)|}$; and it can be thought of a
$\mathcal{U}$-monotonic cycle of length $m$ as a single vertex $u_{i_{1}}$
(for $m=1$), an edge $u_{i_{1}}u_{i_{2}}$ with $i_{1} < i_{2}$ (for $m=2$),
or a cycle $u_{i_{1}}\sim u_{i_{2}}\sim \cdots \sim  u_{i_{m}}\sim u_{i_{1}}$
such that $u_{i_{1}},u_{i_{2}},\ldots , u_{i_{m}}$ appear according to the
linear ordering $u_{1},u_{2},\ldots , u_{|V(G)|}$ (for $m\geq 3$). For a linear ordering
$\mathcal{U}:\{1,2,\ldots , |V(G)|\}  \rightarrow V(G)$, the symbol
$|\mathcal{U}|$ stands for the maximum length of a $\mathcal{U}$-monotonic 
cycle. {\it The circular altitude} of a graph $G$, denoted by $\alpha ^o (G)$, is
defined as the minimum of $|\mathcal{U}|$, where $\mathcal{U}$ ranges over
all linear orderings of $V(G)$. In other words, the circular altitude of a graph $G$
is the maximum positive integer $m$ for which every linear ordering
$\mathcal{U}$ of $V(G)$ admits a $\mathcal{U}$-monotonic cycle of length 
at least $m$.

Let $\mathcal{U}_{1}$ be a linear ordering
of $V(G)$ of the form $u_{1},u_{2},\ldots , u_{|V(G)|}$.
One can correspond this linear ordering to a circle clock
of radius $1$ for which $u_{1},u_{2},\ldots , u_{|V(G)|}$
are arranged clockwise around it with the additional property that
$u_{1}$ is at the highest point of this circle clock. Now, by
rotating all the arranged vertices counterclockwise
$\frac{2\pi }{k}$ radians,
we obtain the linear ordering $\mathcal{U}_{2}$
of $V(G)$ of the form $u_{2},u_{3},\ldots , u_{|V(G)|},u_{1}$.
Continuing in this procedure $k-1$ times, we gain $\mathcal{U}_{k}$ as the linear ordering
of $V(G)$ of the form $u_{k},u_{k+1},\ldots , u_{|V(G)|},u_{1},u_{2},\ldots , u_{k-1}$.
It is obvious that if there exists a 
$\mathcal{U}_{1}$-monotonic cycle of length $m$, 
then for each $k$ in $\{1,2,\ldots , |V(G)|\}$,
there exist $\mathcal{U}_{k}$-monotonic cycles of length $m$, as well.
Therefore, 
$|\mathcal{U}_{1}|=|\mathcal{U}_{2}|=\cdots  =|\mathcal{U}_{|V(G)|}|$.
This fact shows that for determining the circular altitude of $G$,
if we regard an arbitrary vertex $x$ in $V(G)$ as fixed,
we may just restrict our attention to those linear orderings whose first terms are $x$; i.e.,
those linear orderings that begin with $x$.
Also, $|\mathcal{U}_{1}|=|\mathcal{U}_{2}|=\cdots  =|\mathcal{U}_{|V(G)|}|$ implies
that one may just focus on circular orderings of $V(G)$ instead of linear orderings of $V(G)$,
and then consider monotonic cycles in these circular orderings. Nevertheless,
for the sake of simplification, we just use linear orderings rather than circular orderings.

An interesting property of the citcular altitude parameter is its monotonicity.
This parameter is increasing in the sense that if a graph $H$ is a subgraph of a graph $G$,
then $\alpha ^o (H) \leq \alpha ^o (G)$. As a consequence, the clique number
is a lower bound for the circular altitude. So, the inequality
$\omega (G) \leq \alpha ^o (G)$ holds for every graph $G$. Also, 
since the circular altitude is a parameter
related to monotonic cycles, it is a natural proposition that if $\alpha ^o (G) \geq 3$,
then $\alpha ^o (G)$ is greater than or equal to the girth of $G$ \cite{bamberg}.

The concept of circular altitude of graphs was introduced by Cameron in \cite{cameron}
and it was investigated by Bamberg, Corr, Devillers, Hawtin, Pivotto, and Swartz in 
\cite{bamberg}. Bamberg et al. showed that the circular altitude of a graph provides
a lower bound on its circular chromatic number; that is, $\alpha ^o (G) \leq \chi _{c} (G)$.
Also, they provided some examples of graphs with 
$\omega (G) < \alpha ^o (G) < \chi _{c} (G)$.

In this paper, we investigate the circular altitude of graphs. We prove that the circular
altitude of a graph is equal to the maximum of circular altitudes of its blocks. Also,
we show that homomorphically equivalent graphs have the same circular altitudes.
Finally, we prove that the circular altitude of the Cartesian product of two graphs is equal
to the maximum of circular altitudes of its factors.

\section{Reduction to $2$-connected graphs}

This section concerns the relation between the circular
altitude of a graph and circular altitudes of its blocks. We show that for the
circular altitude parameter, we can restrict the study of all graphs just
to the study of their blocks.

\noindent A {\it cut-vertex} of a graph, is a vertex in such a way that removing
it from the graph results in a new graph with a higher number of components.
A graph is called {\it nonseparable} if it is connected and has no cut-vertices.
A {\it block} of a graph $G$ is a nonseparable subgraph of $G$ which is not
a subgraph of another nonseparable subgraph of $G$. In other words,
a block of a graph $G$ is a maximal nonseparable subgraph of $G$.
Every block of a graph is a complete graph with one vertex or is a complete graph
with two vertices or is a $2$-connected graph.

\noindent In the next theorem, we show that the circular
altitude of every graph equals the maximum of circular altitudes of its blocks.
Since $\alpha ^o (K_{1})=1$ and
$\alpha ^o (K_{2})=2$, therefore, the study of circular altitude of graphs
can be reduced to the study of circular altitude of $2$-connected graphs.

\begin{theorem}\label{block}{The circular
altitude of every graph equals the maximum of circular altitudes of its blocks.
}
\end{theorem}

\begin{proof}{
The proof falls into three independent steps, as follows.

\textbf{Step 1.} In this step, we show that the circular
altitude of a graph $G$ is equal to the maximum of circular altitudes of its components.
In this regard, suppose that all components of $G$ are
$G_{1},G_{2},\ldots , G_{m}$. Since $G_{1},G_{2},\ldots , G_{m}$ are subgraphs of
$G$, we have
\begin{center}
$\alpha ^o (G) \geq \max \{\alpha ^o (G_{1}), \alpha ^o (G_{2}), \ldots ,\alpha ^o (G_{m})\}$.
\end{center}

\noindent What is left in this step is to show that
$\alpha ^o (G) \leq \max \{\alpha ^o (G_{1}), \alpha ^o (G_{2}), \ldots ,\alpha ^o (G_{m})\}$.
For each $i$ in $\{1,\ldots , m\}$ let the bijection
$\mathcal{U}_{i}:\{1,2,\ldots , |V(G_{i})|\}  \rightarrow V(G_{i})$, which is considered as
$v_{i_{1}},v_{i_{2}},\ldots, v_{i_{|V(G_{i})|}}$,
be a linear ordering of $V(G_{i})$
in such a way that $\alpha ^o (G_{i})=|\mathcal{U}_{i}|$. Therefore, for each $i$ in $\{1,\ldots ,m\}$,
the maximum length of a $\mathcal{U}_{i}$-monotonic cycle equals $\alpha ^o (G_{i})$.
Now, according to the linear orderings $\mathcal{U}_{1}, \mathcal{U}_{2}, \ldots, \mathcal{U}_{m}$,
regard the natural linear ordering $\mathcal{U}:=\mathcal{U}_{1} \mathcal{U}_{2} \ldots \mathcal{U}_{m}$
as a linear ordering of $V(G)$.
The linear ordering $\mathcal{U}$ satisfies both of the following
two conditions :
\begin{itemize}
\item If $1\leq i < j \leq m$ then each vertex of $G_{i}$ occurs before
each vertex of $G_{j}$.
\item For each $i$ in $\{1,\ldots ,m\}$ and for any two vertices
$a$ and $b$ in $V(G_{i})$, the vertex $a$ occurs before the vertex $b$ in the linear ordering
$\mathcal{U}$ iff $a$ occurs before $b$ in the linear ordering
$\mathcal{U}_{i}$.
\end{itemize}
Since each $\mathcal{U}$-monotonic cycle lies in exactly one of 
$G_{1},G_{2},\ldots , G_{m}$, we obtain
\begin{center}
$|\mathcal{U}| \leq \max \{|\mathcal{U}_{1}|,|\mathcal{U}_{2}|,\ldots ,|\mathcal{U}_{m}|\}$.
\end{center}
Therefore,
\begin{center}
$\alpha ^o (G) \leq |\mathcal{U}| \leq \max \{\alpha ^o (G_{1}) ,\alpha ^o (G_{2}) ,\ldots ,\alpha ^o (G_{m}) \}$.
\end{center}
So, $\alpha ^o (G) = \max \{\alpha ^o (G_{1}) ,\alpha ^o (G_{2}) ,\ldots ,\alpha ^o (G_{m}) \}$;
which is desired in this step.

\textbf{Step 2.} In this step, we suppose that $H$ is a connected graph with a cut-vertex
$x$, and $H_{1},H_{2},\ldots , H_{m}$ are all components of $H-\{x\}$, where
$H-\{x\}$ is the graph obtained by deleting the vertex $x$ from the graph $H$. Also,
for each $i$ in $\{1,\ldots ,m\}$, let $H'_{i}$ be the subgraph of $H$ induced by $V(H_{i})\cup \{x\}$.
In this step, the aim is proving that
$\alpha ^o (H) = \max \{\alpha ^o (H'_{1}) ,\alpha ^o (H'_{2}) ,\ldots ,\alpha ^o (H'_{m}) \}$.
For proving this fact, we should mention that for each of the graphs $H,H'_{1},H'_{2},\ldots ,H'_{m}$,
without loss of generality, 
we can just restrict our attention to those linear orderings that begin with $x$.
For each $i$ in $\{1,2,\ldots ,m\}$, let $\mathcal{W}_{i}$ be a linear ordering of
$V(H'_{i})$ that begins with $x$ and also $\alpha ^o (H'_{i})= |\mathcal{W}_{i}|$.
Now, using $\mathcal{W}_{1},\mathcal{W}_{2},\ldots ,\mathcal{W}_{m}$, we consider
a linear ordering $\mathcal{W}$ of $V(H)$ that
satisfies the following two conditions simultaneously :
\begin{itemize}
\item The linear ordering $\mathcal{W}$ begins with $x$.
\item For each $i$ in
$\{1,2,\ldots ,m\}$ and for any two vertices $a$ and $b$ in $V(H'_{i})$, the vertex $a$ occurs
before the vertex $b$ in the linear ordering $\mathcal{W}$ iff $a$ occurs
before $b$ in the linear ordering $\mathcal{W}_{i}$.
\end{itemize}
Now, since any $\mathcal{W}$-monotonic
cycle whose length is $|\mathcal{W}|$ settles in exactly one of $H'_{1},H'_{2},\ldots , H'_{m}$,
the inequality
$|\mathcal{W}| \leq \max \{|\mathcal{W}_{1}|,|\mathcal{W}_{2}|,\ldots ,|\mathcal{W}_{m}|\}$
holds.
Thus,
\begin{center}
$\alpha ^o (H) \leq |\mathcal{W}| \leq \max \{\alpha ^o (H'_{1}) ,
\alpha ^o (H'_{2}) ,\ldots ,\alpha ^o (H'_{m}) \}$.
\end{center}
On the other hand, since
$H'_{1},H'_{2},\ldots , H'_{m}$ are subgraphs of $H$, we have
\begin{center}
$\alpha ^o (H) \geq \max \{\alpha ^o (H'_{1}) ,
\alpha ^o (H'_{2}) ,\ldots ,\alpha ^o (H'_{m}) \}$.
\end{center}
Hence, 
$\alpha ^o (H) = \max \{\alpha ^o (H'_{1}) ,
\alpha ^o (H'_{2}) ,\ldots ,\alpha ^o (H'_{m}) \}$; and we are done in this step.

We are now in a position to complete the proof of the theorem in Step 3.

\noindent \textbf{Step 3.} Let $\mathcal{G}$ be a graph. We show that
$\alpha ^o (\mathcal{G})= \max \{\alpha ^{o} (\mathcal{B})\ |\ \mathcal{B}\ {\rm is\ a\ block\ of\ } \mathcal{G}\}$.
The inequality 
$\alpha ^o (\mathcal{G}) \geq \max \{\alpha ^{o} (\mathcal{B})\ |\ \mathcal{B}\ {\rm is\ a\ block\ of\ } \mathcal{G}\}$
follows immediately from the fact that each block of $\mathcal{G}$ is a subgraph of $\mathcal{G}$.

\noindent For proving
$\alpha ^o (\mathcal{G})= \max \{\alpha ^{o} (\mathcal{B})\ |\ \mathcal{B}\ {\rm is\ a\ block\ of\ } \mathcal{G}\}$,
it is sufficient to show that there exists a block of $\mathcal{G}$ whose circular altitude equals $\alpha ^o (\mathcal{G})$.
In this regard, define $\mathcal{S}$ as the set of all subgraphs of $\mathcal{G}$ such that their circular altitudes are
equal to $\alpha ^o (\mathcal{G})$. More precisely,
\begin{center}
$\mathcal{S}:=\{\mathcal{T}\ |\ \mathcal{T} {\rm \ is\ a\ subgraph\ of\ } \mathcal{G}\ {\rm and} \ 
\alpha ^o (\mathcal{T})=\alpha ^o (\mathcal{G})\}$.
\end{center}
Let us regard a graph $\mathcal{H}$ in $\mathcal{S}$ with the minimum number of vertices.
Indeed, $\mathcal{H}$ is a subgraph of $\mathcal{G}$ with the minimum number of vertices for
which $\alpha ^o (\mathcal{H})=\alpha ^o (\mathcal{G})$.
According to the Step 1, the graph $\mathcal{H}$ is connected; since otherwise, there exists a component
of $\mathcal{H}$ (with a smaller number of vertices
than to $|V(\mathcal{H})|$) such that its circular altitude equals $\alpha ^o (\mathcal{G})$,
which is a contradiction to the minimality of the number of vertices of $\mathcal{H}$. Also,
on account of
the Step 2, the graph $\mathcal{H}$ has not any cut-vertices. So, $\mathcal{H}$ is a subgraph of $\mathcal{G}$
which is connected and has not any cut-vertices. Consequently, $\mathcal{H}$ is a subgraph of a block
$\mathcal{B}$ of $\mathcal{G}$. Now, since $\mathcal{H}$ is a subgraph of $\mathcal{B}$
and $\mathcal{B}$ is a subgraph of $\mathcal{G}$, we conclude that 
$\alpha ^o (\mathcal{G})=\alpha ^o (\mathcal{H})\leq \alpha ^o (\mathcal{B})\leq \alpha ^o (\mathcal{G})$.
Accordingly, $\alpha ^o (\mathcal{B})=\alpha ^o (\mathcal{G})$; and the assertion follows.
}
\end{proof}

\section{Circular altitude and graph homomorphism}

A {\it homomorphism} from a graph $G$ to a graph $H$, is a function from the vertex set
of $G$ to the vertex set of $H$ which preserves all adjacencies. To be more precise,
a homomorphism from a graph $G$ to a graph $H$ is a function $f:V(G)\rightarrow V(H)$
such that for any two vertices $x$ and $y$ in $V(G)$ the following condition holds
\begin{center}
$xy \in E(G) \Longrightarrow f(x)f(y) \in E(H).$
\end{center}
We write $G\longrightarrow H$ whenever there exists a graph homomorphism from $G$ to $H$.
Two graphs $G$ and $H$ are called {\it homomorphically equivalent} graphs if both of the
conditions $G\longrightarrow H$ and $H\longrightarrow G$ are satisfied.
The theory of graph homomorphisms is one of the most interesting
concepts in graph theory.
Many important notions and parameters in graph theory
can be interpreted in terms of graph homomorphisms. For various
examples, one can see the references \cite{godsil, hahn, hell}. For the circular altitude parameter,
a natural question is whether homomorphically equivalent graphs have the same circular
altitudes. The next theorem gives an affirmative answer to this interesting question.

\begin{theorem}\label{hom}{Homomorphically equivalent graphs have the same 
circular altitudes.
}
\end{theorem}

\begin{proof}{
In order to prove the theorem, it is sufficient to show that if there exists a
surjective homomorphism $f:V(G)\rightarrow V(H)$ from a graph $G$ to a
graph $H$, then $\alpha ^o (G) \leq \alpha ^o (H)$. In this regard, let $\mathcal{U}$
be a linear ordering of $V(H)$ with $|\mathcal{U}|=\alpha ^o (H)$. Consider the
linear ordering $\mathcal{U}$ as $v_{1},v_{2},\ldots, v_{|V(H)|}$. Now, making use of the
linear ordering $\mathcal{U}$, we construct a linear ordering $\widetilde{\mathcal{U}}$
of $V(G)$ as $f^{-1}(v_{1}) f^{-1}(v_{2}) \cdots f^{-1}\left(v_{|V(H)|}\right)$. In the
linear ordering $\widetilde{\mathcal{U}}$,
for each $i$ and $j$ with $1\leq i < j \leq |V(H)|$, each vertex of $f^{-1}(v_{i})$
occurs before each vertex of $f^{-1}(v_{j})$. Indeed, first all vertices of 
$f^{-1}(v_{1}) $ are arbitrarily arranged, then all vertices of 
$f^{-1}(v_{2}) $ are arbitrarily arranged, ... , and finally
all vertices of 
$f^{-1}\left(v_{|V(H)|}\right) $ are arbitrarily arranged in order to construct 
$\widetilde{\mathcal{U}}$. The homomorphism $f:V(G)\rightarrow V(H)$
is a function from $V(G)$ onto $V(H)$ and the restriction of $f$ on the vertices of
any $\widetilde{\mathcal{U}}$-monotonic cycle is an injective function. Furthermore,
by the homomorphism $f$, the image of any $\widetilde{\mathcal{U}}$-monotonic cycle
of length $m$ forms a $\mathcal{U}$-monotonic cycle
of length $m$. Accordingly, $|\widetilde{\mathcal{U}}|\leq |\mathcal{U}|$.
Since $|\mathcal{U}|=\alpha ^o (H)$, we conclude that 
$\alpha ^o (G) \leq |\widetilde{\mathcal{U}}|\leq \alpha ^o (H)$;
and the assertion follows.
}
\end{proof}

In view of the proof of the Theorem \ref{hom}, we observe that
the circular altitude of graphs is an increasing parameter in the sense
that if there exists a homomorphism from a graph $G$ to a graph $H$,
then $\alpha ^o (G) \leq \alpha ^o (H)$. Indeed,
\begin{center}
$G\longrightarrow H\  \ \Longrightarrow \ \ \alpha ^o (G) \leq \alpha ^o (H)$.
\end{center}

Theorem \ref{block} states that the circular
altitude of every graph equals the maximum of circular altitudes of its blocks.
A natural question is the relation between the circular altitude of a graph and the
circular altitudes of its cores. 
\noindent A graph $G$ is called a {\it core graph} whenever every homomorphism from
$G$ to itself is an isomorphism. A {\it core} of a graph $\mathcal{G}$, is a subgraph of
$\mathcal{G}$, say $G$, such that the following two conditions hold simultaneously :
\begin{itemize}
\item $G$ is a core graph,
\item $G$ and $\mathcal{G}$ are homomorphically equivalent.
\end{itemize}
In other words, a core of a graph $\mathcal{G}$ is a subgraph $G$ of $\mathcal{G}$
such that $G$ is a core graph and $\mathcal{G}\longrightarrow G$. It is known that
every graph has a unique core up to isomorphism; see \cite{godsil} or \cite{hahn}
or \cite{hell}. So, Theorem \ref{hom} implies that
the circular altitude of any graph is equal to the circular altitude of its core.
\begin{cor}{The circular altitude of any graph is equal to the circular altitude of its core.
}\end{cor}

\section{Circular altitude of Cartesian product of graphs}

Let $G$ and $H$ be two graphs. {\it The Cartesian product} of $G$ and $H$, denoted
by $G \square H$, is a graph whose vertex set is $V(G) \times V(H)$, and two
vertices $(v_{1},h_{1})$ and $(v_{2},h_{2})$ are adjacent in $G \square H$ if
one of the following two conditions holds :
\begin{itemize}
\item $v_{1}=v_{2}$ and $h_{1}h_{2}\in E(H),$
\item $h_{1}=h_{2}$ and $v_{1}v_{2} \in E(G).$
\end{itemize}
The graphs $G$ and $H$ are called the
{\it factors} of the Cartesian product $G \square H$.
The Cartesian product of graphs is commutative up to isomorphism; that is,
$G \square H$ and $H \square G$ are isomorphic for any
two graphs $G$ and $H$. Also, it is associative up to isomorphism. For more
information about Cartesian product of graphs, one can see \cite{hammack}.

In $1957$, Sabidussi proved that the chromatic number of the Cartesian product
of two graphs is equal
to the maximum of chromatic numbers of its factors \cite{sabidussi}.
More Precisely, for any two graphs
$G$ and $H$, we have $\chi (G \square H)=\max \{\chi (G), \chi (H)\}$.
Zhu showed that this fact also
holds for the circular chromatic number; that is, for any two graphs $G$ and
$H$, the equality $\chi_{c} (G\square H)=\max \{\chi_{c} (G), \chi_{c} (H)\}$ holds \cite{zhu}.
In the next theorem, we are concerned with how circular altitude behaves on Cartesian
product of graphs. We prove an analogue of the two mentioned equalities for the circular
altitude parameter.

\begin{theorem}{The circular altitude of the Cartesian product of two graphs is equal
to the maximum of circular altitudes of its factors. Indeed, for any two graphs
$G$ and $H$ we have $\alpha ^o (G \square H)=\max \{\alpha ^o (G), \alpha ^o (H)\}$.
}
\end{theorem}

\begin{proof}{
Let the vertex sets of graphs $G$ and $H$ be
$V(G):=\{g_{1},g_{2}, \ldots , g_{|V(G)|}\}$ and
$V(H):=\{h_{1},h_{2}, \ldots , h_{|V(H)|}\}$, respectively.
For each integer $i$ in $\{1,2, \ldots , |V(G)|\}$, the subgraph of
$G \square H$ induced by the set $\{g_{i}\} \times V(H)$ is isomorphic to
$H$. Denote this subgraph by $\mathcal{H}_{i}$. This shows that
$\alpha ^o (H) \leq \alpha ^o (G \square H)$. Also,
for each $j$ in $\{1,2, \ldots , |V(H)|\}$, the subgraph of
$G \square H$ induced by the set $V(G) \times \{h_{j}\}$ is isomorphic to
$G$. Similarly, we denote this subgraph by $\mathcal{G}_{j}$. Thus,
$\alpha ^o (G) \leq \alpha ^o (G \square H)$. We deduce that
\begin{center}
$\alpha ^o (G \square H) \geq \max \{\alpha ^o (G), \alpha ^o (H)\}$.
\end{center}
So, we shall have established the theorem
if we prove the following :
\begin{center}
$\alpha ^o (G \square H) \leq \max \{\alpha ^o (G), \alpha ^o (H)\}$.
\end{center}
In this regard, we consider below three cases.

\noindent \textbf{Case 1.} The case that $\alpha ^o (G)=1$.

\noindent In this case, $G$ has no edges and each component of
$G \square H$ is isomorphic to the graph $H$. In view of the
proof of the Theorem \ref{block}, we find that the circular
altitude of every graph equals the maximum of circular altitudes of its components.
This shows that in this case, for the graph
$G \square H$ we have
\begin{center}
$\alpha ^o (G \square H)=\alpha ^o (H)\leq \max \{\alpha ^o (G), \alpha ^o (H)\}$.
\end{center}

\noindent \textbf{Case 2.} The case that $\alpha ^o (H)=1$.

\noindent Since $G \square H$ and $H \square G$ are isomorphic,
similar to the previous case,
\begin{center}
$\alpha ^o (G \square H)=\alpha ^o (H \square G)=\alpha ^o (G)\leq 
\max \{\alpha ^o (G), \alpha ^o (H)\}$.
\end{center}

\noindent \textbf{Case 3.} The case that $\alpha ^o (G)\geq 2$ and $\alpha ^o (H)\geq 2$.

\noindent Without loss of generality, we may assume that
$\mathcal{U}_{G}:=g_{1},g_{2}, \ldots , g_{|V(G)|}$ and
$\mathcal{U}_{H}:=h_{1},h_{2}, \ldots , h_{|V(H)|}$
are linear orderings of $G$ and $H$, respectively, in such a way that
$\alpha ^o (G)=\left|\mathcal{U}_{G}\right|$ and
$\alpha ^o (H)=\left|\mathcal{U}_{H}\right|$. Let us consider the
linear ordering $\mathcal{U}_{G \square H}$
of $V(G \square H)$ for which the vertices of $V(G \square H)$
are ordered lexicographically. More precisely, in the
linear ordering $\mathcal{U}_{G \square H}$,
a vertex $(g_{i},h_{j})$ occurs before a vertex $(g_{i'},h_{j'})$
iff one of the following two conditions holds :
\begin{itemize}
\item $ i<i'$
\item $ i=i' \ {\rm and} \ j<j' \ .$
\end{itemize}
Certainly, the inequality $\alpha ^o (G \square H) \leq \left|\mathcal{U}_{G \square H}\right|$ holds.
Our aim is showing that
$\left|\mathcal{U}_{G \square H}\right|\leq \max \{\alpha ^o (G), \alpha ^o (H)\}$.
The procedure is regarding an arbitrary $\mathcal{U}_{G \square H}$-monotonic cycle
$(g_{i_{1}},h_{j_{1}}),(g_{i_{2}},h_{j_{2}}), \ldots ,(g_{i_{m}},h_{j_{m}})$ as fixed
and proving that $m\leq \max \{\alpha ^o (G), \alpha ^o (H)\}$.

\noindent There is nothing to prove when $m=1$. So, assume that $m\geq 2$.
The vertices of $V(G \square H)$
are ordered lexicographically in the
linear ordering $\mathcal{U}_{G \square H}$. Hence,
$i_{1} \leq i_{2} \leq \cdots \leq i_{m}$.
We show that also
$j_{1} \leq j_{2} \leq \cdots \leq j_{m}$ holds as well. Let $k$ be an integer in
$\{1,2,\ldots ,m-1\}$. Obviously, since $(g_{i_{k}},h_{j_{k}})$ and $(g_{i_{k+1}},h_{j_{k+1}})$
are adjacent in $G \square H$, we have either 
$h_{j_{k}}=h_{j_{k+1}}$
or
$g_{i_{k}}=g_{i_{k+1}}$.
The former implies $j_{k} = j_{k+1}$. The latter imples
$i_{k}=i_{k+1}$; and since 
$(g_{i_{k}},h_{j_{k}})$ occurs before $(g_{i_{k+1}},h_{j_{k+1}})$, we must have
$j_{k} < j_{k+1}$. We conclude that $j_{1} \leq j_{2} \leq \cdots \leq j_{m}$.

\noindent Now, since the sequence $(g_{i_{1}},h_{j_{1}}),(g_{i_{2}},h_{j_{2}}), \ldots ,(g_{i_{m}},h_{j_{m}})$
is a $\mathcal{U}_{G \square H}$-monotonic cycle, two vertices
$(g_{i_{1}},h_{j_{1}})$ and $(g_{i_{m}},h_{j_{m}})$
are adjacent in $G \square H$. Therefore,
one of the following two subcases holds :

\noindent \textbf{Subcase 3.1.} Two vertices $h_{j_{1}}$ and $h_{j_{m}}$
are adjacent in $H$, and $g_{i_{1}}=g_{i_{m}}$.

\noindent In this subcase, since $g_{i_{1}}=g_{i_{m}}$, we have $i_{1}=i_{m}$;
and therefore, $i_{1} = i_{2} = \cdots = i_{m}$. So, $g_{i_{1}}=g_{i_{2}}=\cdots =g_{i_{m}}$;
and thus, the $\mathcal{U}_{G \square H}$-monotonic cycle
$(g_{i_{1}},h_{j_{1}}),(g_{i_{2}},h_{j_{2}}), \ldots ,(g_{i_{m}},h_{j_{m}})$
lies in $\mathcal{H}_{i_{1}}$. Besides, the sequence
$h_{j_{1}},h_{j_{2}}, \ldots ,h_{j_{m}}$ forms a $\mathcal{U}_{H}$-monotonic cycle.
Therefore,
\begin{center}
$m\leq \left|\mathcal{U}_{H}\right|=\alpha ^o (H)\leq \max \{\alpha ^o (G), \alpha ^o (H)\}$.
\end{center}

\noindent \textbf{Subcase 3.2.} Two vertices $g_{i_{1}}$ and $g_{i_{m}}$
are adjacent in $G$, and $h_{j_{1}}=h_{j_{m}}$.

\noindent In this subcase, the equality $h_{j_{1}}=h_{j_{m}}$ implies that $j_{1}=j_{m}$.
So, $j_{1} = j_{2} = \cdots = j_{m}$ and
$h_{j_{1}}=h_{j_{2}}=\cdots =h_{j_{m}}$.
This shows that the $\mathcal{U}_{G \square H}$-monotonic cycle
$(g_{i_{1}},h_{j_{1}}),(g_{i_{2}},h_{j_{2}}), \ldots ,(g_{i_{m}},h_{j_{m}})$
lies in $\mathcal{G}_{j_{1}}$; and also, the sequence
$g_{i_{1}},g_{i_{2}}, \ldots ,g_{i_{m}}$ forms a $\mathcal{U}_{G}$-monotonic cycle.
Accordingly,
\begin{center}
$m\leq \left|\mathcal{U}_{G}\right|=\alpha ^o (G)\leq \max \{\alpha ^o (G), \alpha ^o (H)\}$;
\end{center}
which completes the proof of the theorem.
}
\end{proof}

\end{document}